\documentclass[11pt,leqno]{article}
\usepackage{amsthm,amsmath, mathtools}
\usepackage[round]{natbib}
\usepackage{amsfonts}
\usepackage{ifthen}
\usepackage{enumitem}
\usepackage{amssymb,dsfont,url,color,booktabs}
 \usepackage{tikz}
 \usepackage{setspace}
\usepackage{epstopdf}
\usepackage{chngcntr}
\usepackage{bbm}
\usepackage{array}
\usepackage[export]{adjustbox}[2011/08/13] 
\usepackage{geometry}
\usepackage{rotating}
\usepackage{todonotes}
\usepackage{xspace}
\usepackage[toc,page]{appendix}
 \usepackage[english]{babel}

\usepackage{caption}
\usepackage{subcaption}
\usepackage{multirow}
\usepackage{booktabs}
\usepackage[font={footnotesize}]{caption}

\usepackage{placeins} 		

\theoremstyle{plain}
\newtheorem{theorem}{Theorem}[section]
\newtheorem{lemma}[theorem]{Lemma}
\newtheorem{proposition}[theorem]{Proposition}

\newtheorem{example}[theorem]{Example}

\newcommand{\notiz}[1]{\relax}

\newcommand{\zitep}[1]{\relax}

\newcommand{\1}{\mathds 1}            

\newcommand{\nn}{\mathds N}

\newcommand{\rr}{\mathds R}

\newcommand{\cc}{\mathds C}

\newcommand{\Price}[1][]{
		\ifthenelse{\equal{#1}{}}{\mathit{Price}}{\Price{}^{#1}}
	} 

\newlength{\wordlength}

\newcommand{\olp}[1][]{\textcolor{white}{\underline{\textcolor{black}{\overline x}}_{\hspace{-.05em}\textcolor{black}{#1}}}}

\renewcommand{\cite}{\citet}

\numberwithin{equation}{section}
\numberwithin{figure}{section}
\numberwithin{table}{section}

\begin{document}
\title{\textbf{Improved error bound for multivariate Chebyshev polynomial interpolation}
}

\bigskip
\author{\textbf{Kathrin Glau$\vphantom{l}^{1}$,} \textbf{Mirco Mahlstedt$\vphantom{l}^{1,}$\footnote{The authors thank the KPMG Center of Excellence in Risk Management for their support. We acknowledge fruitful discussions with and feedback from Maximilian Ga{\ss}, Daniel Kressner, Maximilian Mair and Christian P{\"o}tz.}
}
\\\\$\vphantom{l}^{\text{1}}$Technical University of Munich, Germany
}

\maketitle
\begin{abstract}
Chebyshev interpolation is a highly effective, intensively studied method and enjoys excellent numerical properties. The interpolation nodes are known beforehand, implementation is straightforward and the method is numerically stable. For efficiency, a sharp error bound is essential, in particular for high-dimensional applications. For tensorized Chebyshev interpolation, we present an error bound that  improves existing results significantly.
\end{abstract}

	


\section{Introduction}
Tensorized Chebyshev interpolation underlies various algorithms for computational problems in high dimensions. The Chebyshev interpolation of function $f$ is the more beneficial, the higher the cost of evaluating $f$ itself is. The cost of evaluating $f$ directly scales with the computational cost for obtaining the coefficients of the Chebyshev interpolation. For computationally challenging high-dimensional problems, these costs become a bottleneck for the implementation of the interpolation. In these situations it is crucial to use the least number of nodal points possible to achieve a pre-specified accuracy.
 One valuable application is the quantification of parameter uncertainty for high-dimensional integrals that require Monte-Carlo simulations. Here, computationally expensive integrals have to be evaluated for a large set of different parameters. 
At this point interpolation in the parameter space promises to be highly beneficial as shown in \cite{GassGlauMahlstedtMair2016}.

In this paper, we provide an improved error bound for the Chebyshev interpolation of analytic functions. 
\cite{SauterSchwab2004} derive an error bound for the tensorized Chebyshev interpolation. Their proof relies on a method for error estimation for analytic integrands from \cite{davis1975interpolation}.  In \cite{GassGlauMahlstedtMair2016} the result of \cite{SauterSchwab2004} has been slightly improved. The error bound is connected to the radius $\varrho$ of a Bernstein ellipse and in the one-dimensional case \cite{Trefethen2013} presents a different approach which goes back to \cite{Bernstein1912}. In \cite{borm2010efficient} error bounds are presented for the case when the derivatives of function $f$ are bounded. n this paper we assume $f$ to be analytic. We iteratively extend the one-dimensional result shown in \cite{Trefethen2013} to the multivariate by induction over the dimension.  The resulting nested structure of the proof reaches a certain complexity and therefore requires more space than the proof in \cite{SauterSchwab2004}.
Finally, we present the new error bound as a combination of this result with this result from \cite{SauterSchwab2004} and \cite{GassGlauMahlstedtMair2016}. We furthermore discuss examples that show a significant improvement of the new error bound.

In Section \ref{sec-main_result}, we present the main mathematical result and discuss this result. Section \ref{sec-Proofs} provides the proof and finally, Section \ref{sec-conclusion} concludes.


\section{Main result}\label{sec-main_result}
In this section, we provide our main result, the improved error bound for the multivariate Chebyshev interpolation. The main result in Theorem \ref{Asymptotic_error_decay_multidim_combined} is a combination of two error bounds. On the one hand, we use an extension of the result of \cite{SauterSchwab2004} as shown in \cite{GassGlauMahlstedtMair2016}. On the other hand, we extend the one-dimensional result presented in \cite{Trefethen2013} iteratively to the multivariate case.

We consider the tensor based extension of Chebyshev polynomial interpolation of functions $f:\mathcal{X}\rightarrow\mathbb{R}$, $\mathcal{X}=[\underline{x}_{1},\overline{x}_1]\times\ldots \times[\underline{x}_D,\overline{x}_D]\subset\mathbb{R}^D$,
as in e.g. \cite{SauterSchwab2004}.
For notational ease we introduce the polynomials for $\mathcal{X}=[-1,1]^D$ with the obvious extension to general hyperrectangle by the appropriate linear transforms. 
	Let $\overline{N}:=(N_1,\ldots,N_D)$ with $N_i \in\nn_0$ for $i=1,\ldots,D$. The interpolation with $\prod_{i=1}^D (N_{i}+1)$ summands is given by
	\begin{equation}
	I_{\overline{N}}(f)(x) := \sum_{j\in J} c_jT_j(x), 
	\end{equation}
where the function variable $x=(x_1,\dots, x_d)'\in [-1,1]^d$ and 
the summation index $j$ is a multiindex ranging over $J:=\{(j_1,\dots, j_D)\in\nn_0^D: j_i\le N_i\,\text{for }i=1,\ldots,D\}$.
For $j=(j_1,\dots, j_D)\in J$, the basis functions are defined as $T_{j}(x_1,\dots,x_D) = \prod_{i=1}^D T_{j_i}(x_i)$
and the coefficients 
are given by
	\begin{equation}
	\label{def:Chebycj}
		c_j = \Big( \prod_{i=1}^D \frac{2^{\1_{\{0<j_i<N_i\}}}}{N_i}\Big)\sum_{k_1=0}^{N_1}{}^{''}\ldots\sum_{k_D=0}^{N_D}{}^{''} f(x^{(k_1,\dots,k_D)})\prod_{i=1}^D \cos\left(j_i\pi\frac{k_i}{N_i}\right),
	\end{equation}
where $\sum{}^{''}$ indicates that the first and last summand are halved and the Chebyshev nodes $x^k$ for multiindex $k=(k_1,\dots,k_D)\in J$ are given by $x^k = (x_{k_1},\dots,x_{k_D})$
	with the univariate Chebyshev nodes 
	$x_{k_i}=\cos\left(\pi\frac{k_i}{N_i}\right)$ for $k_i=0,\ldots,N_i$ and $i=1,\ldots,D$.

For hyperrectangle $\mathcal{X}\subset\mathbb{R}^D$ and parameter vector $\varrho\in(1,\infty)^D$, we define the textit{generalized Bernstein ellipse} by
\begin{align}\label{eq-genB}
B(\mathcal{X},\varrho):=B([\underline{x}_1,\olp[1]],\varrho_1)\times\ldots\times B([\underline{x}_D,\olp[D]],\varrho_D ),
\end{align}
where $B([\underline{x},\olp],\varrho):=\tau_{[\underline{x},\olp]}\circ B([-1,1],\varrho)$, with the transform $\tau_{[\underline{x},\olp]}\big(\Re(x)\big):=\overline{x} + \frac{\underline{x}-\olp}{2}\big(1-\Re(x)\big)$ and $\tau_{[\underline{x},\olp]}\big(\Im(x)\big):= \frac{\olp-\underline{x}}{2}\Im(x)$ for all $x\in\cc$ and Bernstein ellipses $B([-1,1],\varrho_i)$ for $i=1,\ldots,D$.

\begin{theorem}
 \label{Asymptotic_error_decay_multidim_combined}
  Let $f:\mathcal{X}\rightarrow\mathbb{R}$ have an analytic extension to some generalized Bernstein ellipse $B(\mathcal{X},\varrho)$ for some parameter vector $\varrho\in (1,\infty)^D$ with $\max_{x\in B(\mathcal{X},\varrho)}|f(x)|\le V<\infty$. 
  Then
  \begin{align*}
   \max_{x\in\mathcal{X}}\big|f(x)& - I_{\overline{N}}(f)(x)\big|\le\min\{a(\varrho,N,D),b(\varrho,N,D)\},
  \end{align*}
  where, denoting by $S_D$ the symmetric group on $D$ elements,
  \begin{align*}
  a(\varrho,N,D)&=\min_{\sigma\in S_D}\sum_{i=1}^D 4V\frac{\varrho_{\sigma(i)}^{-N_i}}{\varrho_i-1} + \sum_{k=2}^D 4V\frac{\varrho_{\sigma(k)}^{-N_k}}{\varrho_{\sigma(k)}-1}\cdot 2^{k-1}  \frac{(k-1) + 2^{k-1}-1}{\prod_{j=1}^{k-1}(1-\frac{1}{\varrho_{\sigma(j)}})},\\
  b(\varrho,N,D)&=2^{\frac{D}{2}+1}\cdot V \cdot\left(\sum_{i=1}^D\varrho_i^{-2N_i}\prod_{j=1}^D\frac{1}{1-\varrho_j^{-2}}\right)^{\frac{1}{2}}.
  \end{align*}

  \end{theorem}
  \begin{proof}
  The bound $\max_{x\in\mathcal{X}}\big|f(x) - I_{\overline{N}}(f)(x)\big|\le b(\varrho,N,D)$ follows from \cite[Theorem 2]{GassGlauMahlstedtMair2016} as extension of \cite{SauterSchwab2004}. We show $\max_{x\in\mathcal{X}}\big|f(x) - I_{\overline{N}}(f)(x)\big|\le a(\varrho,N,D)$ in Section \ref{sec-Proofs} in Proposition  \ref{Asymptotic_error_decay_multidim_new_permu}. Combining both results obviously yields the assertion of the theorem.
  \end{proof}
The examples below show that $\min\{a(\varrho,N,D),b(\varrho,N,D)\}$ improves both error bounds $a(\varrho,N,D)$ and $b(\varrho,N,D)$.
Noticing that both bounds are scaled with the factor $V$, we set $V=1$, moreover, we choose $D=2$. 
\begin{example}\label{Example_1}
For $\varrho_1=2.3$ and $\varrho_2=1.8$, and $N_1=N_2=10$, we have $b(\varrho,N,D)=0.0018$ and $a(\varrho,N,D)=0.0066$. Therefore, in this example the error bound $b(\varrho,N,D)$ is sharper.
\end{example} 
\begin{example}\label{Example_2}
If we change slightly the setting from Example \ref{Example_1} to $\varrho_1=2.3$ and $\varrho_2=2.5$, and $N_1=N_2=10$, then the resulting error bounds are $b(\varrho,N,D)=0.0017$ and $a(\varrho,N,D)=0.0011$ and thus, the later is the sharper error bound.
\end{example} 
As shown in Examples \ref{Example_1} and \ref{Example_2}, slight changes in the domain of analyticity and, thus, the radii of the Bernstein ellipses, may reverse the order of $a(\varrho,N,D)$ and $b(\varrho,N,D)$. 
Figure \ref{fig:Same_Rho} displays both error bounds $a(\varrho,N,D)$ and $b(\varrho,N,D)$ for varying $\varrho$ with $\varrho_1=\varrho_2$, $N_1=N_2=10$. We observe that both error bounds intersect at $\varrho_1=\varrho_2\approx2.800882$. For smaller values of $\varrho$, the sharper error bound is $b(\varrho,N,D)$, whereas for higher values $a(\varrho,N,D)$ is sharper.
\begin{figure}[htb!]
\includegraphics[width=0.8\textwidth, center]{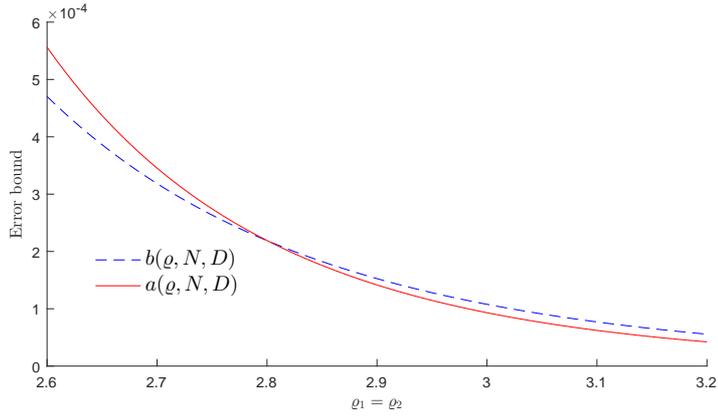}
\caption{Comparison of the error bounds $b(\varrho,N,D)$ (blue, dashed) and $a(\varrho,N,D)$ (red) by setting $\varrho_1=\varrho_2$ and $N_1=N_2=10$. At $\varrho_1=\varrho_2\approx2.800882$ both error bounds intersect.}
\label{fig:Same_Rho}
\end{figure}
So far, the examples indicate that for a smaller radius of the Bernstein ellipse, $b(\varrho,N,D)$ tends to be the better error bound and that for higher radii of the Bernstein ellipses or for strongly differing radii, $a(\varrho,N,D)$ tends to be the sharper error bound.
Our last example shows the situation where thanks to Theorem \ref{Asymptotic_error_decay_multidim_combined} less nodes are required to guarantee a pre-specified accuracy.

\begin{example}\label{Example_3}
Let the radii of the Bernstein ellipse be $\varrho_1=2.95$ and $\varrho_2=9.8$. Assuming $V=1$, we are interested in achieving an accuracy of $\varepsilon\le 2\cdot 10^{-4}$. To achieve $b(\varrho,N,D)=\le\varepsilon$, we have to set $N_1=11$ and $N_2=5$. For achieving $a(\varrho,N,D)\le\varepsilon$, we have to set $N_1=8$ and $N_2=4$. Instead of $72=(11+1)\cdot(5+1)$ nodal points applying error bound $b(\varrho,N,D)$, we only need to use $45=(8+1)\cdot(4+1)$ nodal points applying the error bound $a(\varrho,N,D)$. 
\end{example}
Example \ref{Example_3} highlights the potential of using fewer nodal points to achieve a desired accuracy by comparing both error bounds. Especially when the evaluation of the interpolated function at the nodal points is challenging, this reduces the computational costs noticeably. This particularly arises for Chebyshev interpolation combined with Monte-Carlo simulation for high-dimensional parametric integration as shown in \cite{GassGlauMahlstedtMair2016}.

Summarizing, Theorem \ref{Asymptotic_error_decay_multidim_combined} improves the error bounds $a(\varrho,N,D)$ and $b(\varrho,N,D)$ significantly.

\section{Proofs}\label{sec-Proofs}

In the following, we will present our approach to derive the error bound $a(\varrho,N,D)$ in Theorem \ref{Asymptotic_error_decay_multidim_combined}. Whereas in proof of \cite[Lemma 7.3.3]{SauterSchwab2004} an orthonormal system of appropriately scaled Chebyshev polynomials has been used and each $\varrho_i$ is weighted equally, we will now extend the one-dimensional result in \cite[Theorem 8.2]{Trefethen2013} by induction over the dimension $D$. In each iteration step the interpolation in one additional variable is added consecutively.  

\begin{proposition}
 \label{Asymptotic_error_decay_multidim_new_permu}
Let $f:\mathcal{X}\rightarrow\mathbb{R}$ have an analytic extension to some generalized Bernstein ellipse $B(\mathcal{X},\varrho)$ for some parameter vector $\varrho\in (1,\infty)^D$ with\\
$\max_{x\in B(\mathcal{X},\varrho)}|f(x)|\le V<\infty$.
  Then
  \begin{align*}
 \max_{x\in\mathcal{X}}\big|f(x) - &I_{\overline{N}}(f)(x)\big|
  \\
 &\le\min_{\sigma\in S_D}\sum_{i=1}^D 4V\frac{\varrho_{\sigma(i)}^{-N_i}}{\varrho_i-1} + \sum_{k=2}^D 4V\frac{\varrho_{\sigma(k)}^{-N_k}}{\varrho_{\sigma(k)}-1}\cdot 2^{k-1}  \frac{(k-1) + 2^{k-1}-1}{\prod_{j=1}^{k-1}(1-\frac{1}{\varrho_{\sigma(j)}})},
  \end{align*}
  where $S_D$ denotes the symmetric group on $D$ elements.
  \end{proposition}
%
%
%
 \begin{proof}
We show the statement for an arbitrary $\sigma\in S_D$ and for ease of notation we use $\sigma(i)=i$ for $i=1,\ldots,D$. Obviously,  we can iteratively interpolate in the parameter in such a way that the error bound is minimized by choosing the corresponding $\sigma\in S_D$.
 
We prove the assertion of the theorem via induction over the dimension $D$ of the parameter domain. We assume the function $f$  is analytic in $[-1,1]^D$ and is analytically extendable to the open Bernstein ellipse $B([-1,1]^D,\varrho)$. For $D=1$ and $\mathcal{X}=[-1,1]$ the proof of the assertion is presented in \cite[Theorem 8.2]{Trefethen2013}. The generalization of the assertion to the case of a general parameter interval $\mathcal{X}\subset \rr$ is elementary and follows from a linear transformation as described in \cite[Proof of Theorem 2.2]{GassGlauMahlstedtMair2016}.

The key idea of the proof is to use the triangle inequality to estimate the interpolation error in $D+1$ components via the interpolation error in the $D+1$ component of the original function and the interpolation in the $D+1$ component of the already in $D$ components interpolated function. Hereby, in both cases the issue is basically reduced to an one-dimensional interpolation and the known theory from \cite[Theorem 8.2]{Trefethen2013} can be applied. The crucial step is to derive the bound of the in already in $D$ components interpolated function on the corresponding Bernstein ellipse.

Let us now assume the assertion is proven for dimension $1,\ldots,D$. Let
$\mathcal{X}^{D+1}:=[\underline{x}_1,\olp[1]]\times\ldots\times[\underline{x}_{D+1},\olp[D+1]]$ and let $f:\mathcal{X}\rightarrow\rr$ have an analytic extension to the generalized Bernstein ellipse $B(\mathcal{X}^{D+1},\varrho^{D+1})$ for some parameter vector $\varrho^{D+1}\in (1,\infty)^{D+1}$ and let $\max_{x\in B(\mathcal{X}^{D+1},\varrho^{D+1})}|f(x)|\le V$.
To set up notation, we write $x_1^D=(x_1,\ldots,x_D)$ and define in the following the Chebyshev interpolation operators. For interpolation only in the $i-$th component with $N$ Chebyshev points,
\begin{align*}
I_N^i(f)(x_1^{D+1}):=I_N(f(x_1,\ldots,x_{i-1},\cdot,x_{i+1},\ldots,x_{D+1}))(x_i).
\end{align*}
Analogously, interpolation only in $j$ components with $N_{k_1},\ldots,N_{k_j}$ Chebyshev points is denoted by
\begin{align*}
I_{N_{k_1},\ldots,N_{k_j}}^{j_1,\ldots,j_j}(f)(x_1^{D+1}):=I_{N_{k_j}}^{j_j}\circ\ldots\circ I_{N_{k_1}}^{j_1}(f)(x_1^{D+1}),
\end{align*}
and finally, the interpolation in all $D+1$ components with $N_1,\ldots, N_{D+1}$ Chebyshev points is
\begin{align*}
I_{N_1,\ldots, N_{D+1}}(f)(x_1^{D+1}):=I_{N_{D+1}}^{D+1}\circ\ldots\circ I_{N_{1}}^{1}(f)(x_1^{D+1}).
\end{align*}
In the following the norm $|\cdot|$ denotes the $\infty-$norm on $[-1,1]^{D+1}$. We are interested in the interpolation error
\begin{align*}
&|f(x_1^{D+1}) - I_{N_1,\ldots, N_{D+1}}(f)(x_1^{D+1})|\\
&\quad\quad\quad\le|f(x_1^{D+1}) - I^{D+1}_{N_{D+1}}(f)(x_1^{D+1})|+|I^{D+1}_{N_{D+1}}(f)(x_1^{D+1})-I_{N_1,\ldots, N_{D+1}}(f)(x_1^{D+1})|. 
\end{align*}
We first show that the first part as an one dimensional interpolation is bounded by, \cite[Theorem 8.2]{Trefethen2013},
\begin{align}
|f(x_1^{D+1}) - I^{D+1}_{N_{D+1}}(f)(x_1^{D+1})|\le 4V\frac{\varrho_{D+1}^{-N_{D+1}}}{\varrho_{D+1}-1}.\label{result_1D_d1}
\end{align}
In order to derive \eqref{result_1D_d1}, we have to show that the coefficients of the Chebyshev polynomial interpolation are bounded. Following \cite{Trefethen2013}, the on $x_{D+1}$ depending coefficient $a_{k_{D+1}}$ is defined as
\begin{align*}
a_{k_{D+1}}:=\frac{2^{\mathbbm{1}_{k_{D+1}>0}}}{\pi} \int_{-1}^1\frac{f(x_1^{D+1})T_{k_{D+1}}(p_{D+1})}{\sqrt{1-x_{D+1}^2}}dx_{D+1}.
\end{align*}
By using the same transformation as in the proof of \cite[Theorem 8.1]{Trefethen2013}, just adapted to the multidimensional setting, i.e.
\begin{align*}
x_i&=\frac{z_i+z_i^{-1}}{2},\quad i=1,\ldots D+1,\\
F(z_1,\ldots,z_{D+1})&=F(z_1^{-1},\ldots,z_{D+1}^{-1})=f(x_1,\ldots, x_{D+1}),
\end{align*}
we achieve for the estimation of the coefficient $a_{k_{D+1}}$,
\begin{align*}
|a_{k_{D+1}}|=\left|\frac{2^{-\mathbbm{1}_{k_{D+1}=0}}}{\pi i}\int_{|z_{D+1}|=\varrho_{D+1}}z_{D+1}^{-1-k_{D+1}}F(z_1,\ldots,z_{D+1})dz_{D+1}\right|.
\end{align*}
Here, we use that $F$ is bounded by the same constant as $f$, which is given by assumption, $|f(x_1^{D+1})|_{B([-1,1]^{D+1},\varrho)}\le V$. Therefore, analogously to \cite[Theorem 8.1]{Trefethen2013}, this leads to
\begin{align}
|a_{k_{D+1}}|\le 2\varrho_{D+1}^{-k_{D+1}}V.\label{Coefficient_Ddim}
\end{align}
This estimation can be used to derive \eqref{result_1D_d1} applying \cite[Theorem 8.2]{Trefethen2013}.

For the second part we use
\begin{align*}
|I^{D+1}_{N_{D+1}}(f)(x_1^{D+1})-I_{N_1,\ldots, N_{D+1}}(f)(x_1^{D+1})|=|I^{D+1}_{N_{D+1}}(f-I_{N_1,\ldots, N_{D}}^{1,\ldots,D}(f)(x_1^{D+1}))(x_1^{D+1})|.
\end{align*}
At this point we again apply the triangle inequality and achieve
\begin{align}
|I^{D+1}_{N_{D+1}}&(f-I_{N_1,\ldots, N_{D}}^{1,\ldots,D}(f)(x_1^{D+1}))(x_1^{D+1})|\notag\\ 
&\le|I^{D+1}_{N_{D+1}}(f-I_{N_1,\ldots, N_{D}}^{1,\ldots,D}(f)(x_1^{D+1}))(x_1^{D+1})\label{Induktion_Step1}-(f-I_{N_1,\ldots, N_{D}}^{1,\ldots,D}(f)(x_1^{D+1}))|\\
&\quad+|(f-I_{N_1,\ldots, N_{D}}^{1,\ldots,D}(f)(x_1^{D+1}))|\notag.
\end{align}
The term \eqref{Induktion_Step1} is basically an interpolation in the $D+1$ component of the function $(f-I_{N_1,\ldots, N_{D}}^{1,\ldots,D}(f)(x_1^{D+1}))$. An upper bound $\mathcal{M}(D)$ for this function is given in Lemma \ref{Lemma_Upper_Bound}. With this bound we can estimate the interpolation error of interpolating $(f(x_1^{D+1})-I_{N_1,\ldots, N_{D}}^{1,\ldots,D}(f)(x_1^{D+1}))$ in the component D+1,
\begin{align*}
|I^{D+1}_{N_{D+1}}&(f-I_{N_1,\ldots, N_{D}}^{1,\ldots,D}(f)(x_1^{D+1}))(x_1^{D+1})- (f-I_{N_1,\ldots, N_{D}}^{1,\ldots,D}(f)(x_1^{D+1}))|\\
&\le 4\mathcal{M}(D)\frac{\varrho_{D+1}^{-N_{D+1}}}{\varrho_{D+1}-1}
\end{align*}
 
The term $|(f(x_1^{D+1})-I_{N_1,\ldots, N_{D}}^{1,\ldots,D}(f)(x_1^{D+1}))|$ is the interpolation error in $D$ dimensions and we assume, that this one is by our induction hypothesis bounded, depending on $D$, i.e. 
\begin{align}|(f-I_{N_1,\ldots, N_{D}}^{1,\ldots,D}(f)(x_1^{D+1}))|\le B(D),\quad B(D)>0.\label{Old_error}
\end{align}

Collecting all parts, we achieve for the error of our interpolation in $D+1$ components,
\begin{align*}
|I^{D+1}_{N_{D+1}}(f-I_{N_1,\ldots, N_{D}}^{1,\ldots,D}(f)(x_1^{D+1}))(x_1^{D+1})|\le4V\frac{\varrho_{D+1}^{-N_{D+1}}}{\varrho_{D+1}-1}+B(D)+4\mathcal{M}(D)\frac{\varrho_{D+1}^{-N_{D+1}}}{\varrho_{D+1}-1}.
\end{align*}
Finally, if we start with $D=1$ and apply the presented procedure step-wise,  we get via straightforward induction ,
\begin{align*}
&B(D)=\sum_{i=1}^D 4V\frac{\varrho_i^{-N_i}}{\varrho_i-1} + \sum_{k=2}^D 4\mathcal{M}(k-1)\frac{\varrho_k^{-N_k}}{\varrho_k-1}.
\end{align*}
Naturally, we can further estimate the error by using $\frac{s_i}{\varrho_i}<1$ and resp. $(1-\frac{s_i}{\varrho_i})<1$ in the numerator,
\begin{align*}
B(D)&\le \sum_{i=1}^D 4V\frac{\varrho_i^{-N_i}}{\varrho_i-1} + \sum_{k=2}^D 4V\frac{\varrho_k^{-N_k}}{\varrho_k-1}\cdot 2^{k-1}  \frac{(k-1) + 2^{k-1}-1}{\prod_{j=1}^{k-1}(1-\frac{s_j}{\varrho_j})}.
\end{align*}
Recalling the definition of $s_i=1+\epsilon$ with $\epsilon\in(0,\min_{j=1}^D \varrho_j -1)$, the definition holds for any $\epsilon\in(0,\min_{j=1}^D\varrho_j -1)$ and therefore also for $\lim_{\epsilon\to 0}$
\begin{align*}
B(D)\le&\lim_{\epsilon\to 0}\sum_{i=1}^D 4V\frac{\varrho_i^{-N_i}}{\varrho_i-1} + \sum_{k=2}^D 4V\frac{\varrho_k^{-N_k}}{\varrho_k-1}\cdot 2^{k-1}  \frac{(k-1) + 2^{k-1}-1}{\prod_{j=1}^{k-1}(1-\frac{1+\epsilon}{\varrho_j})}\\
=&\sum_{i=1}^D 4V\frac{\varrho_i^{-N_i}}{\varrho_i-1} + \sum_{k=2}^D 4V\frac{\varrho_k^{-N_k}}{\varrho_k-1}\cdot 2^{k-1}  \frac{(k-1) + 2^{k-1}-1}{\prod_{j=1}^{k-1}(1-\frac{1}{\varrho_j})}.
\end{align*}
\end{proof}
In the following lemmata, we use the following notation $x_1^M=(x_1,\ldots,x_M)$ and the convention $\frac{N}{0}=\infty,\ N\in\mathbb{N}^{+}$.
\begin{lemma}
\label{Lemma_Trefethen_Polynom_1D}
 Let $\mathcal{X}\ni x_1^M \mapsto f(x_1^M)$ be a real valued function that has an analytic extension to some generalized Bernstein ellipse $B(\mathcal{X},\varrho)$ for some parameter vector $\varrho\in (1,\infty)^{M}$.\\
 Then the Chebyshev polynomial interpolation $I_N^1(f)(x_1^M)$ is given by,
 \begin{align}
 I_N^1(f)(x_1^M)&=\sum_{k=0}^{N}a_k(x_2^M) T_k(x_1)+\sum_{k=N+1}^{\infty} a_k(x_2^M) T_{m(k,N)}(x_1),\label{Cheby_Series_Error1}
 \end{align}
 where $m(k,N)=|(k+N-1)(mod2N)-(N-1)|$ and $a_k(x_2^M)=\frac{2}{\pi}\int_{-1}^1 f(x_1^M)\frac{T_k(x_1)}{\sqrt{1-x_1^2}}dx_1$
\end{lemma}
\begin{proof}
Following \cite[Equation (4.9)]{Trefethen2013}, from aliasing properties of Chebyshev polynomials  it results that 
\begin{align*}
f(x_1^M)-I_N^1(f)(x_1^M)=\sum_{k=N+1}^{\infty} a_k(x_2^M) (T_k(x_1)-T_{m(k,N)}(x_1)).
\end{align*}
By writing the Chebyshev series for $f(x_1^M)$, see \cite{Trefethen2013}, we get,
\begin{align*}
&\sum_{k=0}^{\infty}a_k(x_2^M) T_k(x_1)-I_N^1(f)(x_1^M)=\sum_{k=N+1}^{\infty} a_k(x_2^M) (T_k(x_1)-T_{m(k,N)}(x_1)),
\end{align*}
and rearranging terms yields \eqref{Cheby_Series_Error1}.
\end{proof}
\begin{lemma}\label{Lemma_Upper_Bound}
 Let $\mathcal{X}\ni x_1^M \mapsto f(x_1^{D+1})$ be a real valued function that has an analytic extension to some generalized Bernstein ellipse $B(\mathcal{X},\varrho)$ for some parameter vector $\varrho\in (1,\infty)^{D+1}$. Then
 \begin{align*}
&\sup_{x_{D+1}\in B([-1,1],\varrho_{D+1})}|f(x_1^{D+1})-I_{N_1,\ldots, N_{D}}^{1,\ldots,D}(f)(x_1^{D+1})|\le \mathcal{M}(D)\\
&:=2^DV  \frac{\sum_{i=1}^D \left(\frac{s_i}{\varrho_i}\right)^{Ni+1}+\sum_{\sigma\in\{0,1\}^D \setminus\{0\}^D}\prod_{\delta:\sigma_{\delta}=0}(1-\left(\frac{s_{\delta}}{\varrho_{\delta}}\right)^{N_{\delta}+1} \prod_{\delta:\sigma_{\delta}=1}\left(\frac{s_{\delta}}{\varrho_{\delta}}\right)^{N_{\delta}+1}}{\prod_{j=1}^D(1-\frac{s_j}{\varrho_j})}
 \end{align*}
\end{lemma}
\begin{proof}
Starting with,
\begin{align*}
\sup_{x_{D+1}\in B([-1,1],\varrho_{D+1})}|f(x_1^{D+1})-I_{N_1,\ldots, N_{D}}^{1,\ldots,D}(f)(x_1^{D+1})|,
\end{align*}
we express the interpolation of $f$ in $D$ components as in Lemma \ref{Lemma_Interpolation_D_Components},
\begin{align*}
\sup_{x_{D+1}\in B([-1,1],\varrho_{D+1})}\bigg|f(x_1^{D+1})
-\sum_{\sigma\in\{0,1\}^D}\sum_{\delta=1}^D\sum_{k_{\delta}=(N_{\delta}+1)\cdot\sigma_{\delta}}^{\frac{N_{\delta}}{1-\sigma_{\delta}}}I(k_1^{D},x_{D+1})\tau(k_1^{D},\sigma_1^{D},x_1^{D})\bigg|.
\end{align*}
Following \cite{Trefethen2013} and as used in Lemma \ref{Lemma_Trefethen_Polynom_1D}, we can express $f$ in the following way,
\begin{align*}
f(x_1^{D+1})=\sum_{\delta=1}^D \sum_{k_{\delta}=0}^{\infty}I(k_1^{D},x_{D+1})\tau(k_1^{D},\sigma_1^{D}=0,x_1^{D}),
\end{align*}
leading to,
\begin{align*}
\sup_{x_{D+1}\in B([-1,1],\varrho_{D+1})}&|f(x_1^{D+1})-I_{N_1,\ldots, N_{D}}^{1,\ldots,D}(f)(x_1^{D+1})|\\
&=\sup_{x_{D+1}\in B([-1,1],\varrho_{D+1})}\bigg|\sum_{\delta=1}^D \sum_{k_{\delta}=0}^{\infty}I(k_1^{D},x_{D+1})\tau(k_1^{D},\sigma_1^{D}=0,x_1^{D})\\
&\quad-\sum_{\sigma\in\{0,1\}^D}\sum_{\delta=1}^D\sum_{k_{\delta}=(N_{\delta}+1)\cdot\sigma_{\delta}}^{\frac{N_{\delta}}{1-\sigma_{\delta}}}I(k_1^{D},x_{D+1})\tau(k_1^{D},\sigma_1^{D},x_1^{D})\bigg|.
\end{align*}
In the next step, we use from the second summand the part $\sigma=\{0\}^D$, subtract it from the subtrahend and use the triangle inequality. 
\begin{align*}
&\sup_{x_{D+1}\in B([-1,1],\varrho_{D+1})}|f(x_1^{D+1})-I_{N_1,\ldots, N_{D}}^{1,\ldots,D}(f)(x_1^{D+1})|\\
&=\sup_{x_{D+1}\in B([-1,1],\varrho_{D+1})}\bigg|\sum_{i=1}^D \left(\sum_{k_i=N_i+1}^{\infty}\sum_{j=1,j\neq i}^D\sum_{k_{j}=0}^{\infty}I(k_1^{D},x_{D+1})\tau(k_1^{D},\sigma_1^{D}=0,x_1^{D})\right) \\
&\quad-\sum_{\sigma\in\{0,1\}^D \setminus\{0\}^D}\sum_{\delta=1}^D\sum_{k_{\delta}=(N_{\delta}+1)\cdot\sigma_{\delta}}^{\frac{N_{\delta}}{1-\sigma_{\delta}}}I(k_1^{D},x_{D+1})\tau(k_1^{D},\sigma_1^{D},x_1^{D})\bigg|\\
&\le \sup_{x_{D+1}\in B([-1,1],\varrho_{D+1})}\bigg|\sum_{i=1}^D \left(\sum_{k_i=N_i+1}^{\infty}\sum_{j=1,j\neq i}^D\sum_{k_{j}=0}^{\infty}I(k_1^{D},x_{D+1})\tau(k_1^{D},\sigma_1^{D}=0,x_1^{D})\right)\bigg|\\
&+\quad\bigg|\sum_{\sigma\in\{0,1\}^D \setminus\{0\}^D}\sum_{\delta=1}^D\sum_{k_{\delta}=(N_{\delta}+1)\cdot\sigma_{\delta}}^{\frac{N_{\delta}}{1-\sigma_{\delta}}}I(k_1^{D},x_{D+1})\tau(k_1^{D},\sigma_1^{D},x_1^{D})\bigg|
\end{align*}
To estimate the supremum, we first need estimations for $|I(k_1^{D},x_{D+1})|$ and \\$|\tau(k_1^{D},\sigma_1^{D},x_1^{D})|$. 
\begin{align*}
|I(k_1^{D},&x_{D+1})|=\bigg|\prod_{i=1}^D\frac{2^{\mathbbm{1}_{k_i>0}}}{\pi}\int_{[-1,1]^D}f(x_1^{D+1})\prod_{j=1}^D\frac{T_{k_j}(x_j)}{\sqrt{1-x_j^2}}d(x_1^D)\bigg|\\
=&\bigg|\prod_{i=2}^D\frac{2^{\mathbbm{1}_{k_i>0}}}{\pi}\int_{[-1,1]^{D-1}}\frac{2^{\mathbbm{1}_{k_1>0}}}{\pi}\int_{-1}^1f\frac{T_{k_1}(x_1)}{\sqrt{1-x_1^2}}d(x_1)\prod_{j=2}^D\frac{T_{k_j}(x_j)}{\sqrt{1-x_j^2}}d(x_2^D)\bigg|. 
\end{align*}
Analogously to deriving the estimation \eqref{Coefficient_Ddim}, we can estimate the integral with respect to $x_1$ as $\frac{2^{\mathbbm{1}_{k_1>0}}}{\pi}\int_{-1}^1f\frac{T_{k_1}(x_1)}{\sqrt{1-x_1^2}}d(p_1)\le 2V\varrho_1^{-k_1}$. The remaining $D-1$ dimensional integral can in a similar way be estimated as $D-1$ one-dimensional integrals with $V=1$. Altogether, this results in the following estimation for $|I(k_1,\ldots,k_D)|$,
\begin{align*}
|I(k_1^{D},x_{D+1})|\le 2^D V \prod_{i=1}^D\varrho_{i}^{-k_i}.
\end{align*}
For $|\tau(k_1^{D},\sigma_1^{D}=0,x_1^{D})|$, we make use of Bernstein's inequality, using that the norm of each Chebyshev polynomial is bounded by 1 on $[-1,1]$. For each $i=1,\ldots,D$ we choose a Bernstein ellipse with radius $s_i$ such that $1<s_i<\varrho_i$. Here, we define $s_i=1+\epsilon$ and this yields for $x:\ x_i\in B([-1,1],s_i),\ i=1,\ldots,D$,
\begin{align*}
|\tau(k_1^{D},\sigma_1^{D},x_1^{D})|=\prod_{\delta:\sigma_{\delta}=0}T_{k_{\delta}}(x_{\delta})\prod_{\delta:\sigma_{\delta}=1}T_{m_{\delta}(k_{\delta})}(x_{\delta})\le\prod_{\delta:\sigma_{\delta}=0}s_{\delta}^{k_{\delta}}\prod_{\delta:\sigma_{\delta}=1}s_{\delta}^{m_{\delta}(k_{\delta})}.
\end{align*}

By definition, it holds $m_{\delta}(k_{\delta})\le k_{\delta}$. This leads to
\begin{align*}
|\tau(k_1^{D},\sigma_1^{D}=0,x_1^{D})|\le&\prod_{i=1}^D s_i^{k_i}.
\end{align*}
Using both estimates leads to
\begin{align*}
\sup_{x_{D+1}\in B([-1,1],\varrho_{D+1})}&|f(x_1^{D+1})-I_{N_1,\ldots, N_{D}}^{1,\ldots,D}(f)(x_1^{D+1})|\\
&\le \sup_{x_{D+1}\in B([-1,1],\varrho_{D+1})}\bigg|\sum_{i=1}^D \left(\sum_{k_i=N_i+1}^{\infty}\sum_{j=1,j\neq i}^D\sum_{k_{j}=0}^{\infty}2^DV\prod_{l=1}^D\left(\frac{s_l}{\varrho_l}\right)^{k_l} \right)\bigg|\\
&\quad\quad +\bigg|\sum_{\sigma\in\{0,1\}^D \setminus\{0\}^D}\sum_{\delta=1}^D\sum_{k_{\delta}=(N_{\delta}+1)\cdot\sigma_{\delta}}^{\frac{N_{\delta}}{1-\sigma_{\delta}}}2^DV\prod_{l=1}^D\left(\frac{s_l}{\varrho_l}\right)^{k_l}\bigg|.
\end{align*}
Due to $s_i<\varrho_i$ we can apply the convergence results for the geometric series. This leads to
\begin{align*}
&\sup_{x_{D+1}\in B([-1,1],\varrho_{D+1})}|f(x_1^{D+1})-I_{N_1,\ldots, N_{D}}^{1,\ldots,D}(f)(x_1^{D+1})|\\
&\quad\quad\quad\quad\le\mathcal{M}(D):= \sup_{x_{D+1}\in B([-1,1],\varrho_{D+1})} \bigg|2^DV\sum_{i=1}^D\frac{\left(\frac{s_i}{\varrho_i}\right)^{Ni+1}}{\prod_{j=1}^D(1-\frac{s_j}{\varrho_j})}\bigg|\\
&\quad\quad\quad\quad\quad +\bigg|2^DV \sum_{\sigma\in\{0,1\}^D \setminus\{0\}^D} \frac{\prod_{\delta:\sigma_{\delta}=0}(1-\left(\frac{s_{\delta}}{\varrho_{\delta}}\right)^{N_{\delta}+1} \prod_{\delta:\sigma_{\delta}=1}\left(\frac{s_{\delta}}{\varrho_{\delta}}\right)^{N_{\delta}+1}}{\prod_{j=1}^D(1-\frac{s_j}{\varrho_j})}\bigg|\\
&=2^DV  \frac{\sum_{i=1}^D \left(\frac{s_i}{\varrho_i}\right)^{Ni+1}+\sum_{\sigma\in\{0,1\}^D \setminus\{0\}^D}\prod_{\delta:\sigma_{\delta}=0}(1-\left(\frac{s_{\delta}}{\varrho_{\delta}}\right)^{N_{\delta}+1} \prod_{\delta:\sigma_{\delta}=1}\left(\frac{s_{\delta}}{\varrho_{\delta}}\right)^{N_{\delta}+1}}{\prod_{j=1}^D(1-\frac{s_j}{\varrho_j})}.
\end{align*}
\end{proof}

\begin{lemma}
\label{Lemma_Interpolation_D_Components}
 Let $\mathcal{X}\ni x_1^M \mapsto f(x_1^M)$ be a real valued function that has an analytic extension to some generalized Bernstein ellipse $B(\mathcal{X},\varrho)$ for some parameter vector $\varrho\in (1,\infty)^{M}$. For$D\le M$ let
 \begin{align*}
I(k_1^{D},x_{D+1}^{M})&=\prod_{i=1}^D\frac{2^{\mathbbm{1}_{k_i>0}}}{\pi}\int_{[-1,1]^D}f(x_1^{M})\prod_{j=1}^D\frac{T_{k_j}(x_j)}{\sqrt{1-x_j^2}}d(x_1,\ldots,x_D),\\
\tau(k_1^{D},\sigma_1^{D},x_1^{D})&=\prod_{\delta:\sigma_{\delta}=0}T_{k_{\delta}}(x_{\delta})\prod_{\delta:\sigma_{\delta}=1}T_{m_{\delta}}(x_{\delta}),
\end{align*}
then the interpolation of $f(x_1^{M})$ in $D$ components is given by:
\begin{align*}
I_{N_1,\ldots,N_D}^{1,\ldots,D}(f)(x_1^M)=\sum_{\sigma\in\{0,1\}^D}\sum_{\delta=1}^D\sum_{k_{\delta}=(N_{\delta}+1)\cdot\sigma_{\delta}}^{\frac{N_{\delta}}{1-\sigma_{\delta}}}I(k_1^{D},x_{D+1}^{M})\tau(k_1^{D},\sigma_1^{D},x_1^{D}).
\end{align*}
\end{lemma}
\begin{proof}
We proof this lemma via induction over the dimension $D$. For $D=1$ it follows from Lemma \ref{Lemma_Trefethen_Polynom_1D},
\begin{align*}
I_{N_1}^{1}(f)(x_1^M)=&\sum_{k_1=0}^{N_1}\frac{2^{\mathbbm{1}_{k_1>0}}}{\pi}\int_{[-1,1]}f(x_1^{M})\frac{T_{k_1}(x_1)}{\sqrt{1-x_1^2}}dx_1 T_{k_1}(x_1)\\
&\quad\quad+\sum_{k_1=N_1+1}^{\infty}\frac{2^{\mathbbm{1}_{k_1>0}}}{\pi}\int_{[-1,1]}f(x_1^{M})\frac{T_{k_1}(x_1)}{\sqrt{1-x_1^2}}dx_1 T_{m_1}(x_1).
\end{align*}
Embedded in the introduced notation we get for $D=1$,
\begin{align*}
I_{N_1}^{1}(f)(x_1^M)=\sum_{\sigma\in\{0,1\}}\sum_{\delta=1}^1\sum_{k_{\delta}=(N_{\delta}+1)\cdot\sigma_{\delta}}^{\frac{N_{\delta}}{1-\sigma_{\delta}}}I(k_1^1,x_2^M)\tau(k_1^1,\sigma_1^1,x_1^1).
\end{align*}
For the induction step from $D-1$ to $D$, we assume the interpolation in $D-1$ components is given by
\begin{align*}
I_{N_1,\ldots,N_{D-1}}^{1,\ldots,{D-1}}(f)(x_1^M)=\sum_{\sigma\in\{0,1\}^{D-1}}\sum_{\delta=1}^{D-1}\sum_{k_{\delta}=(N_{\delta}+1)\cdot\sigma_{\delta}}^{\frac{N_{\delta}}{1-\sigma_{\delta}}}I(k_1^{D-1},x_{D}^M)\tau(k_1^{D-1},\sigma_1^{D-1},x_1^{D-1}).
\end{align*}
For the interpolation in $D$ components we make use of
\begin{align*}
I^{1,\ldots,D}_{N_1,\ldots, N_{D}}(f)(x_1^{M})=I_{N_{D}}^{D}\circ\ldots\circ I_{N_{1}}^{1}(f)(x_1^{M})=I_{N_{D}}^{D}\circ I_{N_1,\ldots,N_{D-1}}^{1,\ldots,{D-1}}(f)(x_1^M).
\end{align*}
As for $D=1$ we apply \cite[p.27]{Trefethen2013} and this leads to
\begin{align*}
I_{N_1,\ldots, N_{D}}(f)(x_1^{D})=&\sum_{k_D=0}^{N_D}\frac{2^{\mathbbm{1}_{k_D>0}}}{\pi}\int_{-1}^1 I_{N_1,\ldots,N_{D-1}}^{1,\ldots,{D-1}}(f)(x_1^M) \frac{T_{k_D}(x_D)}{\sqrt{1-x_D^2}}dx_D T_{k_D}(x_D)\\
+&\sum_{k_D=N_D+1}^{\infty}\frac{2^{\mathbbm{1}_{k_D>0}}}{\pi}\int_{-1}^1 I_{N_1,\ldots,N_{D-1}}^{1,\ldots,{D-1}}(f)(x_1^M) \frac{T_{k_D}(x_D)}{\sqrt{1-x_D^2}}dx_D T_{m_D}(x_D).
\end{align*}
By the induction hypothesis and the definitions of $I(k_1^{D-1},x_D^M)$ and\\ $\tau(k_1^{D-1},\sigma_1^{D-1},x_1^{D-1})$, we achieve,
\begin{align*}
&\int_{-1}^1 I_{N_1,\ldots,N_{D-1}}^{1,\ldots,{D-1}}(f)(x_1^M) \frac{T_{k_D}(x_D)}{\sqrt{1-x_D^2}}dx_D\\
&\quad=\int_{[-1,1]} \sum_{\sigma\in\{0,1\}^{D-1}}\sum_{\delta=1}^{D-1}\sum_{k_{\delta}=(N_{\delta}+1)\cdot\sigma_{\delta}}^{\frac{N_{\delta}}{1-\sigma_{\delta}}}I(k_1^{D-1},x_D^M)\tau(k_1^{D-1},\sigma_1^{D-1},x_1^{D-1}) \frac{T_{k_D}(x_D)}{\sqrt{1-x_D^2}}dx_D\\
&\quad=\int_{[-1,1]} \sum_{\sigma\in\{0,1\}^{D-1}}\sum_{\delta=1}^{D-1}\sum_{k_{\delta}=(N_{\delta}+1)\cdot\sigma_{\delta}}^{\frac{N_{\delta}}{1-\sigma_{\delta}}}\prod_{i=1}^{D-1}\frac{2^{\mathbbm{1}_{k_i>0}}}{\pi}\\
&\quad\quad\quad \int_{[-1,1]^{D-1}}f(x_1^M)\prod_{j=1}^{D-1}\frac{T_{k_j}(x_j)}{\sqrt{1-x_j^2}}d(x_1,\ldots,x_{D-1})\frac{T_{k_D}(x_D)}{\sqrt{1-x_D^2}}dx_D.
\end{align*} 

Rearranging terms yields,
\begin{align*}
I_{N_1,\ldots, N_{D}}(f)(x_1^{M})=&\sum_{k_D=0}^{N_D}\sum_{\sigma\in\{0,1\}^{D-1}}\sum_{\delta=1}^{D-1}\sum_{k_{\delta}=(N_{\delta}+1)\cdot\sigma_{\delta}}^{\frac{N_{\delta}}{1-\sigma_{\delta}}}\prod_{i=1}^{D}\frac{2^{\mathbbm{1}_{k_i>0}}}{\pi}\int_{[-1,1]^{D}}f(x_1^M)\\
&\prod_{j=1}^D\frac{T_{k_j}(x_j)}{\sqrt{1-x_j^2}}d(x_1^{D})\prod_{\delta:\sigma_{\delta}=0}T_{k_{\delta}}(x_{\delta})\prod_{\delta:\sigma_{\delta}=1}T_{m_{\delta}}(x_{\delta}) T_{k_D}(x_D)\\
&+\sum_{k_D=N_D+1}^{\infty}\sum_{\sigma\in\{0,1\}^{D-1}}\sum_{\delta=1}^{D-1}\sum_{k_{\delta}=(N_{\delta}+1)\cdot\sigma_{\delta}}^{\frac{N_{\delta}}{1-\sigma_{\delta}}}\prod_{i=1}^{D}\frac{2^{\mathbbm{1}_{k_i>0}}}{\pi}\int_{[-1,1]^{D}}f(p_1^M)\\
&\prod_{j=1}^D\frac{T_{k_j}(x_j)}{\sqrt{1-x_j^2}}d(x_1^{D})\prod_{\delta:\sigma_{\delta}=0}T_{k_{\delta}}(x_{\delta})\prod_{\delta:\sigma_{\delta}=1}T_{m_{\delta}}(x_{\delta}) T_{m_D}(x_D).\\
\end{align*}
This can be expressed as
\begin{align}
&I_{N_1,\ldots,N_D}^{1,\ldots,D}(f)(x_1^M)=\sum_{\sigma\in\{0,1\}^D}\sum_{\delta=1}^D\sum_{k_{\delta}=(N_{\delta}+1)\cdot\sigma_{\delta}}^{\frac{N_{\delta}}{1-\sigma_{\delta}}}I(k_1^D,x_{D+1}^M)\tau(k_1^D,\sigma_1^D,x_1^D).
\end{align}
\end{proof}


\section{Conclusion}\label{sec-conclusion} 
In this article, we have provided an enhanced error bound for tensorized Chebyshev polynomial interpolation in Theorem \ref{Asymptotic_error_decay_multidim_combined} and have shown several examples. Example \ref{Example_3} highlights the effect of the improved error bound. Here, less interpolation nodes are required to guarantee a pre-specified accuracy. This significantly reduces the computational time, especially if the evaluation of function $f$ at the nodal points is time-consuming.

\bibliographystyle{chicago}
  \bibliography{LiteraturFourierRB}

\end{document}